\documentclass[11pt]{amsart}  

\usepackage{amsmath}
\usepackage{amsthm}
\usepackage{amsfonts}
\usepackage{amssymb}
\usepackage{eucal}
\usepackage{yfonts}
\usepackage[all]{xy}
\usepackage{amsxtra}
\usepackage{appendix} 
\usepackage{tensor} 
\usepackage{url}
\usepackage{upgreek}
\usepackage{etoolbox}
\apptocmd{\sloppy}{\hbadness 10000\relax}{}{}

\usepackage[pdftex, breaklinks, linktocpage=true, bookmarksopen=true,bookmarksopenlevel=0,bookmarksnumbered=true]{hyperref}

\RequirePackage{color}
\definecolor{bwgreen}{rgb}{0.183,1,0.5}
\definecolor{bwmagenta}{rgb}{0.7,0.0,0.1}
\definecolor{bwblue}{rgb}{0.317,0.161,1}

\hypersetup{
pdftitle = {On the Atkin operator in characteristic p},
pdfauthor = {\textcopyright\ Bryden Cais},
pdfcreator = {\LaTeX\ with package \flqq hyperref\frqq},
colorlinks = {true},
linkcolor={bwmagenta},	 
anchorcolor={},	 
citecolor={bwblue},	
filecolor={},	
menucolor={},	
runcolor={},	
urlcolor={bwgreen}, 
}

\CompileMatrices
\entrymodifiers={+!!<0pt,\fontdimen22\textfont2>}

\DeclareFontFamily{OT1}{rsfs}{}
\DeclareFontShape{OT1}{rsfs}{n}{it}{<-> rsfs10}{}
\DeclareMathAlphabet{\mathscr}{OT1}{rsfs}{n}{it}

\DeclareFontFamily{OT1}{pzc}{}
\DeclareFontShape{OT1}{pzc}{n}{it}{<->s*[2.2]pzc}{}
\DeclareMathAlphabet{\mathpzc}{OT1}{pzc}{b}{sl}

\makeatletter

\newcommand{\Rmnum}[1]{\expandafter\@slowromancap\romannumeral #1@}
\makeatother

\DeclareMathOperator{\ord}{ord}

\newcommand*{\pr}{\rho}

\DeclareMathOperator{\GL}{GL}
\DeclareMathOperator{\SL}{SL}

\DeclareMathOperator{\Ig}{Ig}

\DeclareMathOperator{\Ell}{Ell}

\DeclareMathOperator{\cusps}{cusps}

\newcommand*{\Z}{\ensuremath{\mathbf{Z}}}

\renewcommand*{\k}{\kappa}

\newcommand*{\C}{\mathbf{C}}
\newcommand*{\E}{\mathscr{E}}     
\newcommand*{\F}{\mathbf{F}}

\newcommand*{\I}{\mathscr{I}}                               

\renewcommand*{\L}{\mathscr{L}}

\renewcommand*{\O}{\mathscr{O}}

\newcommand*{\scrP}{\mathscr{P}}

\renewcommand*{\int}{\ensuremath{\mathrm{int}}}

\renewcommand*{\SS}{\ensuremath{{\mathrm{ss}}}}

\renewcommand*{\o}[1]{\overline{#1}}

\newcommand*{\wt}[1]{\widetilde{#1}}

\newcommand*{\can}{\text{-}\mathrm{can}}

\theoremstyle{plain}
  \newtheorem{theorem}{Theorem}
  \newtheorem{proposition}[theorem]{Proposition}
  
  \newtheorem{corollary}[theorem]{Corollary}

\theoremstyle{definition}

\theoremstyle{remark}

\include{header}

\numberwithin{theorem}{section}  
\numberwithin{equation}{section}

\begin{document}
\title{On the $U_p$ operator in characteristic $p$}

\author{Bryden Cais}
\address{University of Arizona, Tucson}
\curraddr{Department of Mathematics, 617 N. Santa Rita Ave., Tucson AZ. 85721}
\email{cais@math.arizona.edu}

\thanks{
	During the writing of this paper, the author was partially supported by an NSA Young Investigator grant
	(H98230-12-1-0238).
	We are very grateful to David Zureick-Brown for many helpful conversations.
	}

\date{\today}

\begin{abstract}
	For a perfect field $\k$ of characteristic $p>0$,
	a positive ingeger $N$ not divisible by $p$, 
	and an arbitrary subgroup $\Gamma$ of $\GL_2(\Z/N\Z)$,
	we prove (with mild additional hypotheses when $p\le 3$)
	that the $U$-operator on the space $M_k(\scrP_{\Gamma}/\k)$
	of (Katz) modular forms for $\Gamma$ over $\k$ induces a surjection 
	$U:M_{k}(\scrP_{\Gamma}/\k)\twoheadrightarrow M_{k'}(\scrP_{\Gamma}/\k)$
	for all $k\ge p+2$, where $k'=(k-k_0)/p + k_0$ with $2\le k_0\le  p+1$
	the unique integer congruent to $k$ modulo $p$.  When $\k=\F_p$, 
	$p\ge 5$, $N\neq 2,3$, and 
	$\Gamma$  is the subgroup of upper-triangular or 
	upper-triangular unipotent matrices,
	this recovers a recent result of Dewar \cite{Dewar}. 
\end{abstract}

\maketitle

\section{Introduction}\label{intro}

Fix a prime $p$, 
an integer $N>0$ with $p\nmid N$, and a 
subgroup $\Gamma$ of $\GL_2(\Z/N\Z)$.
Let $\wt{\Gamma}$ be the preimage in $\SL_2(\Z)$ of $\Gamma_0:=\Gamma\cap \SL_2(\Z/N\Z)$,
and write $\wt{M}_k(\wt{\Gamma})$ for the space of
weight $k$ mod $p$ modular forms for $\wt{\Gamma}$
(in the sense of Serre \cite[\S1.2]{Serre}).
When $N=1$, a classical result of Serre \cite[\S2.2, Th\'ero\`eme 6]{Serre}
asserts that the $U_p$ operator is a contraction: for $k\ge p+2$,
the map $U_p: \wt{M}_k(\Gamma(1))\rightarrow \wt{M}_k(\Gamma(1))$
factors through the subspace $\wt{M}_{k'}(\Gamma(1))$ for some $k'<k$
satisfying $pk' \le k+ p^2-1$. In fact, Serre's result may be generalized and significantly sharpened:

\begin{theorem}\label{Main}
	Let $\k$ be a perfect field of characteristic $p$ and
	denote by $M_k(\scrP_{\Gamma}/\k)$ the space of weight $k$ Katz modular forms for 
		$\Gamma$ over $\k$ $($see $\S\ref{Geometry}$$)$.
	Let $k_0$ be 
	the unique integer between $2$ and $p+1$ congruent to $k$ modulo $p$, and 
	if $p\le 3$, assume that $N > 4$ and that 
	$\Gamma_0$ is a subgroup 
	of the upper-triangular unipotent matrices.
	Then for $k\ge p+2$, the $U$-operator $($see $\S\ref{Geometry}$$)$ acting on 
	$M_k(\scrP_{\Gamma}/\k)$ induces a surjection
	$U:M_k(\scrP_{\Gamma}/\k)\twoheadrightarrow M_{k'}(\scrP_{\Gamma}/\k)$, for $k':=(k-k_0)/p +k_0$.
\end{theorem}	

When $\wt{\Gamma}=\Gamma_{\star}(N)$ for $\star=0,1$ and $\k=\F_p$, the endomorphism
$U$ coincides with the usual Atkin operator $U_p$ (see Corollary \ref{ClassicalUp}).
In particular, if $p\ge 5$, so $\wt{M}_k(\wt{\Gamma})\simeq M_{k}(\scrP_{\Gamma}/\F_p)$ 
(by Theorems 1.7.1, and 1.8.1--1.8.2 of \cite{Katz}) 
and $N\neq 2,3$, Theorem \ref{Main} is due to Dewar \cite{Dewar}.
Both Serre's original result and Dewar's refinement of it rely on 
a delicate analysis of 
the interplay between the operators $U_p$, $V_p$, and $\theta$ acting on mod $p$ modular forms. In the present note, we take an algebro-geometric perspective, and show how Theorem \ref{Main}
follows
immediately from a (trivial extension of a) 
general theorem of Tango\footnote{Tango's paper,
which 
appeared the year prior to Serre's \cite{Serre}, is perhaps not as well-known as
it should be.} \cite{Tango} on the behavior of vector bundles under the Frobenius map.  
In this optic, the contractivity of $U_p$ in characteristic $p$ is simply an instance
of the ``Dwork Principle" of analytic continuation along Frobenius.
In particular, we use neither the $\theta$-operator, nor the notion of ``filtration" of a mod $p$ modular form.

\section{Tango's Theorem}
Fix a perfect field $\k$ of characteristic $p$, and write $\sigma:\k\rightarrow \k$
for the $p$-power Frobenius automorphism of $\k$.
Let $X$ be a smooth, proper, and geometrically connected curve over $\k$
of genus $g$.  Attached to $X$ is its {\em Tango number}:
\begin{equation}
	n(X) := \max\left\{ \sum_{x\in X(\o{\k})} \left\lfloor \frac{\ord_{x}(df)}{p}\right\rfloor\ : 
	f\in \o{\k}(X)\setminus \o{\k}(X)^p \right\},
\end{equation}
where $\o{\k}(X)$ is the function field of $X_{\o{\k}}$.  As in Lemma 10 and Proposition 14 of
\cite{Tango}, it is easy to see that $n(X)$ is well-defined and is an integer satisfying
$-1\le n(X)\le \lfloor (2g-2)/p\rfloor$, with the lower bound an equality if and only if $g=0$.

\begin{proposition}[Tango]\label{TangoThm}
	Let $S\neq X$ be a reduced closed subscheme of $X$ with corresponding ideal sheaf $\I_S\subseteq \O_X$,
	and let $\L$ be a line bundle on $X$.  If $\deg \L > n(X)$ then the natural $\sigma$-linear map
	\begin{equation}
		\xymatrix{
			{F^*:H^1(X,\L^{-1}\otimes \I_S)} \ar[r] & {H^1(X,\L^{-p}\otimes \I_S)}
			}\label{Frob}
	\end{equation} 
	induced by pullback by the absolute Frobenius of $X$ is injective,
	and
	the natural $\sigma^{-1}$-linear ``trace map" 
	\begin{equation}
		\xymatrix{
			{F_*:H^0(X,\Omega^1_{X/\k}(S)\otimes\L^p)} \ar[r] & {H^0(X,\Omega^1_{X/\k}(S)\otimes\L)}
			}\label{Cartier}
	\end{equation} 
	given by the Cartier operator $($\cite{CartierNouvelle}, \cite[\S10]{SerreTopology}$)$ is surjective.
\end{proposition}

\begin{proof}
	First note that the formation of $(\ref{Frob})$ and $(\ref{Cartier})$ is compatible,
	via $\sigma$- (respectively $\sigma^{-1}$-) linear extension,
	with any scalar extension $\k\rightarrow \k'$ to a perfect field $\k'$; we may therefore
	assume that $\k$ is algebraically closed.
	As the two assertions are dual\footnote{Note that $\kappa$-linear duality
	interchanges $\sigma$-linear maps with $\sigma^{-1}$-linear ones.} by Serre duality \cite[\S10, Proposition 9]{SerreTopology}, it suffices to prove the injectivity of (\ref{Frob}).
	The case $S=\emptyset$ is Tango's Theorem\footnote{Strictly speaking, Tango requires $g>0$;  however, by tracing through Tango's argument---or by direct calculation---one sees easily that 
	the result holds when $g=0$ as well.} \cite[Theorem 15]{Tango}. We reduce the general case to this one 
	as follows: using that
	$\deg(\L) >0$ and that $\O_X/\I_S^j$ is a skyscraper sheaf for all $j >0$, one finds a commutative diagram
	with exact rows
	\begin{equation*}
	\xymatrix{
		0 \ar[r] & {H^0(X,\O_X/\I_S)} \ar[r]\ar[d]_-{F^*} & {H^1(X,\L^{-1}\otimes \I_S)} \ar[r]\ar[d]_-{F^*} & 
		{H^1(X,\L^{-1})} \ar[r]\ar[d]_-{F^*} & 0\\
		0 \ar[r] & {H^0(X,\O_X/\I_S^p)} \ar[r]\ar[d] & {H^1(X,\L^{-p}\otimes \I_S^p)} \ar[r]\ar[d] &
		 {H^1(X,\L^{-p})} \ar[r]\ar@{=}[d] & 0\\
		0 \ar[r] & {H^0(X,\O_X/\I_S)} \ar[r] & {H^1(X,\L^{-p}\otimes \I_S)} \ar[r] & {H^1(X,\L^{-p})} \ar[r] & 0
	}
	\end{equation*}
	in which the lower vertical arrows are induced by the inclusion of ideal sheaves $\I_S^p\subseteq \I_S$.
	Using that $\k=\o{\k}$ and identifying $H^0(X,\O_X/\I_S)$ with $\k^{S}$, the left vertical composite is easily seen to coincide
	with the map $\oplus_S \sigma:\k^S\rightarrow \k^S$ which is $\sigma$ on each factor; it is therefore
	injective.  As the right vertical composite map is injective by Tango's Theorem, an easy diagram chase
	finishes the proof.
\end{proof}

\section{Modular forms mod \texorpdfstring{$p$}{p} as differentials on the Igusa curve}\label{Geometry}

In order to apply Tango's Theorem to prove Theorem \ref{Main},
we must recall Katz's geometric definition of mod $p$
modular forms, and Serre's interpretation of them as certain meromorphic
differentials on the Igusa curve.  

Let us write\footnote{Here, we follow the notation of \cite[\S9.4]{KM}: By definition
$\Z[\zeta_N]$ is the finite free $\Z$-algebra
$\Z[X]/\Phi_N(X)$, where $\Phi_N$ is the $N$-th cyclotomic polynomial and
$\zeta_N$ corresponds to $X$, equipped with its natural Galois action of $(\Z/N\Z)^{\times}$.} $R_{\Gamma}:=(\Z[\zeta_N])^{\det(\Gamma)}$,
and for any $R_{\Gamma}$-algebra $A$ denote by $\scrP_{\Gamma}/A$
the moduli problem $([\Gamma(N)]/\Gamma)^{R_{\Gamma}\can}\otimes_{R_{\Gamma}} A$
on $(\Ell/A)$ (see \S3.1, \S7.1, 9.4.2, and 10.4.2 of \cite{KM})
and by $M_k(\scrP_{\Gamma}/A)$ the space of weight $k$ Katz modular forms for $\scrP_{\Gamma}/A$
({\em e.g.} \cite[\S6]{Ulmer}) that are holomorphic at $\infty$ in the sense of \cite[\S1.2]{Katz}.
Equivalently, $M_k(\scrP_{\Gamma}/A)$ is the $A$-submodule of level $N$, weight $k$
modular forms in the sense of \cite[\Rmnum{7}.3.6]{DR} that are invariant under
the natural action of $\Gamma_0$.
Viewing $\C$ as an $R_{\Gamma}$-algebra via $\zeta_N\mapsto \exp(2\pi i/N)$,
we remark that $M_k(\scrP_\Gamma/\C)$
{\em is} the ``classical" space of weight $k$ modular forms for $\wt{\Gamma}$ over $\C$ defined
via the trancendental theory \cite[\Rmnum{7}.4]{DR}.

Now fix a ring homomorphism $R_{\Gamma}\rightarrow \kappa$ with $\kappa$ a perfect field of
characteristic $p$.  From here until the end of the section we will assume that 
$\scrP_{\Gamma}/\kappa$ is representable and that $-1$
acts without fixed points on the space of cusp-labels for $\Gamma$
(see \cite[\S10.6]{KM} and {\em c.f.} \cite[10.13.7--8]{KM}).  We will later
explain how to relax these hypotheses to those of Theorem \ref{Main}.
We write $Y_\Gamma$ (respectively $X_\Gamma$) for the associated (compactified) moduli scheme;
by \cite[10.13.12]{KM}, one knows that $X_\Gamma$ is a proper, smooth, and geometrically 
connected curve over $\kappa$.  Writing $\pr:\E\rightarrow Y_\Gamma$ for the universal elliptic
curve, our hypothesis that $-1$ acts without fixed points ensures that the line bundle
$\omega_{\Gamma}:=\pr_*\Omega^1_{\E/Y_\Gamma}$ on $Y_\Gamma$ admits a canonical 
extension, again denoted $\omega_{\Gamma}$, to a line bundle on $X_\Gamma$ \cite[10.13.4, 10.13.7]{KM}.
By definition, $M_k(\scrP_\Gamma/\k) = H^0(X_\Gamma,\omega_{\Gamma}^{k})$.

Let $I_{\Gamma}$ be the {\em Igusa curve of level $p$ over $X_\Gamma$}; by definition,
$I_{\Gamma}$ is the compactified moduli scheme associated to the simultaneous problem
$[\scrP_{\Gamma}/\k,[\Ig(p)]]$ on $(\Ell/\k)$ \cite[\S12]{KM}. 
By \cite[12.7.2]{KM}, the Igusa curve is proper, smooth, and geometrically connected,
and the natural map 
$\pi:I_{\Gamma}\rightarrow X_\Gamma$,
is finite \'etale and Galois with group $(\Z/p\Z)^{\times}$
outside the supersingular points,  and totally ramified over every supersingular 
point.  We define $\omega:=\pi^*\omega_{\Gamma}$, and recall \cite[12.8.2--3]{KM} that 
there is
a canonical section $a\in H^0(I_\Gamma,\omega)$ which has $q$-expansion
equal to 1, vanishes to order 1
at each supersingular point, 
and on which $d\in (\Z/p\Z)^{\times}$ acts (via its action on $I_\Gamma$) through $\chi^{-1}$, 
for $\chi:(\Z/p\Z)^{\times}=\F_p^{\times}\hookrightarrow \F_p$ the mod $p$ Teichm\"uller
character.
The following is a straightforward generalization of a theorem of Serre;
see \cite[\S12.8]{KM} and {\em c.f.} Propositions 5.7--5.10 of \cite{tameness}.

\begin{proposition}\label{MFInter}
	Fix an integer $k\ge 2$ and let $k_0\le k$ be any integer with $2\le k_0 \le p+1$.  
	The map $f\mapsto f/a^{k_0-2}$ induces an natural isomorphism of $\k$-vector spaces
	\begin{equation}
		M_k(\scrP_\Gamma/\k)\simeq H^0(I_\Gamma,\Omega^1_{I_\Gamma/\k}(\cusps + \delta_{k_0}\cdot\SS)\otimes\omega^{k-k_0})(\chi^{k_0-2}),
		\label{IgusaMF}
	\end{equation}
	where $\delta_{k_0}=1$ when $k_0=p+1$ and is zero otherwise; here,
	$\SS$, $\cusps$ are the reduced supersingular and cuspidal divisors, respectively.
\end{proposition}

\begin{proof}
	The proof is a straightforward adaptation of Propositions 5.7--5.10 of \cite{tameness};
	for the convenience of the reader, we sketch the argument.  Thanks to \cite[10.13.11]{KM},
	the Kodaira-Spencer map \cite[10.13.10]{KM} provides an isomorphism of line bundles 
	$\omega^{2}_{\Gamma}\simeq \Omega^1_{X_\Gamma/\k}(\cusps)$
	on $X_\Gamma$ which, after pullback along $\pi$, gives an isomorphism
	\begin{equation}
		\omega^{2}\simeq \Omega^1_{I_\Gamma/\k}(-(p-2)\SS + \cusps)\label{KS}
	\end{equation} 
	of line bundles
	on $I_\Gamma$ as $\pi$ is \'etale outside $\SS$ and totally (tamely) ramified at each
	supersingular point. 

	Since $a\in H^0(I_{\Gamma},\omega)$ has simple zeroes at the supersingular points,
	via (\ref{KS}) any global section $f$ of $\omega^{k}_{\Gamma}$ 
	induces a global section $\pi^*f/a^{k_0-2}$ of $\Omega^1_{I_\Gamma/\k}(\cusps+\delta_{k_0}\cdot\SS)\otimes \omega^{k-k_0}$
	on which $(\Z/p\Z)^{\times}$ acts through $\chi^{k_0-2}$; thus the map (\ref{IgusaMF})
	is well-defined.  
	Since the $q$-expansion of $a$ is $1$ and $I_{\Gamma}$ is geometrically connected, 
	the $q$-expansion principle then shows that (\ref{IgusaMF}) is injective.
	To prove surjectivity, observe that by (\ref{KS}), a global section of 
	$\Omega^1_{I_\Gamma/\k}(\cusps+\delta_{k_0}\cdot\SS)\otimes \omega^{k-k_0}$ gives a meromorphic section $h$ of 
	$\omega^{k-k_0+2}$ satisfying $\ord_x(h) \ge -(p-1)$ at each supersingular point $x$, 
	with equality possible only when $k_0=p+1$.
	If $h$ lies in the $(k_0-2)$-eigenspace of the action of $(\Z/p\Z)^{\times}$, then 
	$f:=a^{k_0-2}h$ descends to a meromorphic section of $\omega^{k}_{\Gamma}$ over $X_\Gamma$
	satisfying 
	$$(p-1)\ord_x(f) = \ord_{x}(h) + k_0 -2 \ge k_0-p-1$$
	at each supersingular point $x\in X_{\Gamma}(\o{\k})$,
	with equality possible only when $k_0=p+1$.
	Since the left side is a multiple of $p-1$ and $k_0 \ge 2$, we must have $\ord_x(f) \ge 0$
	in all cases, and $f$ is a global (holomorphic) section of $\omega^{k}_{\Gamma}$ over $X_\Gamma$ 
	with $\pi^*f/a^{k_0-2}=h$.
\end{proof}

Using Proposition \ref{MFInter}, the Cartier operator $F_*$ on meromorphic differentials
induces, by ``transport of structure", a $\sigma^{-1}$-linear endomorphism
$U:M_k(\scrP_{\Gamma}/\k)\rightarrow M_k(\scrP_{\Gamma}/\k)$.  If $G$
is any group of automorphisms of $X(\Gamma)$, then the action of $G$ commutes with $F_*$
(ultimately because the $p$-power map in characteristic $p$ commutes with all ring homomorphisms),
and we likewise obtain a $\sigma^{-1}$-linear endomorphism $U$ of $M_k(\scrP_{\Gamma}/\k)^G$. 
This allows us to define $U$ even when $\scrP_{\Gamma}/\k$ is not representable as follows.
Choose a prime $\ell > 3N$, and let
$\Gamma'$ be the unique subgroup of $\GL_2(\Z/N\ell\Z)$ projecting to the trivial
subgroup of $\GL_2(\Z/\ell\Z)$ and to $\Gamma$ in $\GL_2(\Z/N\Z)$.
Then for any perfect field $\k'$ of characteristic $p$ admitting a map from $R_{\Gamma'}$,
the moduli problem $\scrP_{\Gamma'}/\k'$ is representable, there is a natural action of 
$G:=\SL_2(\Z/\ell\Z)$ on $M_k(\scrP_{\Gamma'}/\k')$, and one has
$M_k(\scrP_{\Gamma}/\k')=M_k(\scrP_{\Gamma'}/\k')^G$ ({\em c.f.} \cite[\Rmnum{7}.3.3]{DR} and
\cite[\S1.2]{Katz}).  
Since $M_k(\scrP_{\Gamma}/\k)\otimes_{\k}\k'\simeq M_k(\scrP_{\Gamma}/\k')$,
we obtain the desired endomorphism $U$ of $M_k(\scrP_{\Gamma}/\k)$ by descent,
and it is straightforward to check that it is independent of our initial choices of $\ell$ and $\k'$.
By post-composition with the $\sigma$-linear isomorphism\footnote{Explicitly, 
this isomorphism sends $f\in M_k(\scrP_{\Gamma}/\k)$ to the modular form $f^{\sigma}$
defined by $f^{\sigma}(E,\alpha):=(E^{\sigma},\alpha^{\sigma})$} 
$M_k(\scrP_{\Gamma}/\k)\simeq M_k(\scrP_{\Gamma}^{\sigma^{-1}}/\k)$ induced by the ``exotic isomorphism" of moduli problems 
$\scrP_{\Gamma}/\k \simeq \scrP_{\Gamma}^{\sigma^{-1}}/\k$ \cite[12.10.1]{KM} we obtain a $\k$-linear
map $U^{\#}:M_k(\scrP_{\Gamma}/\k)\rightarrow M_k(\scrP_{\Gamma}^{\sigma^{-1}}/\k)$.
When $\scrP_{\Gamma}$ is defined over\footnote{A sufficient condition for this to happen
is that $\det(\Gamma)$ contain the residue class of $p\bmod N$.} 
$\F_p$ in the sense that $R_{\Gamma}$
admits a (necessarily unique) surjection to $\F_p$, one has canonically 
$\scrP_{\Gamma}/\F_p = \scrP_{\Gamma}^{\sigma^{-1}}/\F_p$ as problems on $(\Ell/\F_p)$,
and $U^{\#}$ is an endomorphism of $M_k(\scrP_{\Gamma}/\F_p)$.
The maps $U$ and $U^{\#}$ are natural generalizations of Atkin's $U_p$-operator:

\begin{proposition}\label{UpRel}
	Suppose that $\scrP_{\Gamma}/\k$ is representable and
	let $c$ be any cusp of $X(\Gamma)$ defined over $\k$.  Then $q^{1/e}$
	is a uniformizing parameter at $c$ for some divisor $e$ of $N$, and
	for any $f\in M_k(\scrP_{\Gamma}/\k)$, the formal expansions of 
	$Uf$ at $c$ and of $U^{\#}f$ at $c^{\sigma^{-1}}$ are given by
	\begin{equation*}
		Uf = \sum_{n\ge 0}\sigma^{-1}(a_{np}) q^{n/e}\quad\text{and}\quad
		U^{\#}f = \sum_{n\ge 0}a_{np} q^{n/e}\ \text{respectively, where}\ f = \sum_{n\ge 0} a_n q^{n/e}.
	\end{equation*}	
\end{proposition}

\begin{proof}
	Using the well-known local description of the Cartier operator on meromorphic 
	differentials ({\em e.g.} \cite[\S10, Proposition 8]{SerreTopology}), the result
	follows easily from the arguments of Propositions 2.8 and 5.7 of \cite{tameness};
	see also (the proof of) \cite[Proposition 5.9]{tameness}.  
\end{proof}

\begin{corollary}\label{ClassicalUp}
	Suppose that $\wt{\Gamma} = \Gamma_{\star}(N)$ for $\star=0,1$.
	Then $R_{\Gamma}=\Z$ and the resulting endomorphisms
	$U$ and $U^{\#}$ of $M_k(\scrP_{\Gamma}/\F_p)$ coincide with the Atkin operator $U_p$,
	whether or not $\scrP_{\Gamma}/\F_p$ is representable.
\end{corollary}

\begin{proof}
	That $R_{\Gamma}=\Z$ is clear, as $\det(\Gamma)=(\Z/N\Z)^{\times}$.
	By the discussion above, we may reduce to the representable case, and the result then
	follows from Proposition \ref{UpRel} and the $q$-expansion principle.
\end{proof}	

\section{Proof of Theorem \ref{Main}}\label{ThPf}

We now prove Theorem \ref{Main}. Fix $k$ and let $k_0$ and $k'$ be 
as in the statement of Theorem \ref{Main}.
First suppose that $\scrP_{\Gamma}\otimes_{R_{\Gamma}}\k$ is representable and that $-1$
acts without fixed points on the cusp-labels of $\Gamma$.
Using (\ref{KS}) and the fact that $a$ has simple zeroes along $\SS$ we compute ({\em c.f.} \cite[12.9.4]{KM})
\begin{equation*}
	\deg \omega = \frac{2g-2}{p} + \frac{1}{p}\deg(\cusps) > \left\lfloor \frac{2g-2}{p}\right\rfloor
	\ge n(I_\Gamma)
\end{equation*}
where $g$ is the genus of $I_{\Gamma}$.
Applying Proposition \ref{TangoThm} with $X=I_\Gamma$, $S=\cusps+\delta_{k_0}\cdot\SS$, 
and $\L=\omega$, we conclude from (\ref{Cartier}) and the relation $k-k_0=p(k'-k_0)$ that the Cartier operator
\begin{equation*}
	\xymatrix{
		{F_*:H^0(I_\Gamma,\Omega^1_{I_\Gamma/\k}(\cusps+\delta_{k_0}\cdot\SS)\otimes \omega^{k-k_0})} \ar[r] &
		{H^0(I_\Gamma),\Omega^1_{I_\Gamma/\k}(\cusps+\delta_{k_0}\cdot\SS)\otimes \omega^{k'-k_0})}	
	}
\end{equation*}
is surjective whenever $k-k_0 \ge p$.  Passing to $\chi^{k_0-2}$-eigenspaces for $(\Z/p\Z)^{\times}$ and appealing to
Proposition \ref{MFInter} and Corollary \ref{ClassicalUp} then completes the proof in this case.  

Now when $p\le 3$, the hypotheses $N>4$ and $\wt{\Gamma}\subseteq \Gamma_1(N)$ of Theorem \ref{Main}
ensure that $\scrP_{\Gamma}\otimes_R\k$ is representable (as it maps to the moduli problem $[\Gamma_1(N)]$,
which is representable for $N\ge 4$ by \cite[10.9.6]{KM})
and that $-1$ acts without fixed points on the cusp-labels of $\Gamma$ \cite[10.7.4]{KM}.
If $p\ge 5$, we may choose a prime $\ell > 3N$ with $\ell\not\equiv 0,\pm 1\bmod p$, so that
$p\nmid|\SL_2(\Z/\ell\Z)|$.  Then for $N':=N\ell$ and 
$\Gamma':=1\times \Gamma \subseteq \SL_2(\Z/\ell\Z)\times \SL_2(\Z/N\Z)=\SL_2(\Z/N\ell\Z)$, we have 
(after passing to an appropriate extension $\k'$ of $\k$) that
$\scrP_{\Gamma'}\otimes_{R_{\Gamma'}}\k'$ is representable with $-1$ acting freely on the cusp-labels of $\Gamma'$ \cite[10.7.1, 10.7.3]{KM}.
We conclude that the $U$-operator induces a surjection
of $\kappa[\SL_2(\Z/\ell\Z)]$-modules
$M_{k}(\scrP_{\Gamma'}/\k')\twoheadrightarrow M_{k'}(\scrP_{\Gamma'}/\k')$.  
Our choice of $\ell$ ensures that the ring $\kappa[\SL_2(\Z/\ell\Z)]$
is semisimple, so passing to $\SL_2(\Z/\ell\Z)$-invarants is exact.  As the space of $\SL_2(\Z/\ell\Z)$-invariant 
weight $k$ modular forms for $\Gamma'$ coincides with $M_{k}(\scrP_{\Gamma}/\k')$ ({\em c.f.} the definition of $U$ in \S\ref{Geometry}), passing to invariants and descending from $\k'$ to $\k$
then completes the proof of Theorem \ref{Main} in the general case.

\bibliographystyle{amsalpha_noMR}
\bibliography{mybib}
\end{document}